\documentclass[sn-mathphys]{sn-jnl}% Math and Physical Sciences Reference Style
\newcommand{\calB}{\mathcal{B}}

\newcommand{\calF}{\mathcal{F}}
\newcommand{\calG}{\mathcal{G}}

\newcommand{\bbC}{\mathbb{C}}

\newcommand{\bbN}{\mathbb{N}}

\newcommand{\bbR}{\mathbb{R}}
\newcommand{\bbZ}{\mathbb{Z}}

\newcommand{\rme}{\mathrm{e}}
\newcommand{\rmi}{\mathrm{i}}

\newcommand{\abs}[1]{\left\lvert #1 \right\rvert}

\DeclareMathOperator{\supp}{supp}

\renewcommand{\Re}{\operatorname{Re}}
\renewcommand{\Im}{\operatorname{Im}}

\newcommand{\dd}{\,\mathrm{d}}

\jyear{2022}%

\theoremstyle{thmstyleone}%
\newtheorem{theorem}{Theorem}%  meant for continuous numbers
\newtheorem{lemma}[theorem]{Lemma}

\newtheorem{corollary}[theorem]{Corollary}

\theoremstyle{thmstyletwo}%
\newtheorem{remark}{Remark}%

\theoremstyle{thmstylethree}%

\begin{document}

\title[Sampling at twice the Nyquist rate in two frequency bins guarantees ...]{Sampling at twice the Nyquist rate in two frequency bins guarantees uniqueness in Gabor phase retrieval}


\author*[1]{\fnm{Matthias} \sur{Wellershoff}}\email{wellersm@umd.edu}



\affil*[1]{\orgdiv{Department of Mathematics}, \orgname{University of Maryland}, \orgaddress{\street{4176 Campus Drive}, \city{College Park, MD}, \postcode{20742}, \country{USA}}}

% \affil[2]{\orgdiv{Department}, \orgname{Organization}, \orgaddress{\street{Street}, \city{City}, \postcode{10587}, \state{State}, \country{Country}}}

% \affil[3]{\orgdiv{Department}, \orgname{Organization}, \orgaddress{\street{Street}, \city{City}, \postcode{610101}, \state{State}, \country{Country}}}


\abstract{We show that bandlimited signals can be uniquely recovered (up to a constant global phase factor) from Gabor transform magnitudes sampled at twice the Nyquist rate in two frequency bins.}




\maketitle

\bmhead{Acknowledgments}

The author would like to thank A.~Bandeira for posing the question which is being answered in the present letter. In addition, the author would like to thank the reviewers for their comments which allowed for a significant improvement in the presentation of the paper. Finally, the author acknowledges funding through the Swiss National Science Foundation grant 200021\_184698.

% Then, double check whether reviewer 2 will be happy with the new product. Then, double check whether reviewer 1 will be happy with the new product. Upload to Arxiv. Delete part of the long remark to make everything shorter. Then, write a new cover letter and resubmit (ask Philippe if the Springer Nature Template is necessary/convenient for them bc. I do not like to work with it).

\section{Introduction}

We consider the recovery of square-integrable signals $f \in L^2(\bbR)$ from the magnitude of their \emph{Gabor transforms} 
\begin{equation}\label{eq:Gabor_transform}
    \calG f (x,\omega) := \sqrt[4]{2} \int_\bbR f(t) \rme^{-\pi (t-x)^2} \rme^{-2\pi\rmi t \omega}, \qquad (x,\omega) \in \bbR^2,
\end{equation}
also called the \emph{Gabor phase retrieval problem}. Specifically, we are interested in questions related to the reconstruction of $f$ from $\abs{\calG f}$ measured on lattices\footnote{A \emph{lattice} $\Lambda \subset \bbR^2$ is a discrete subset of the time-frequency plane that can be written as $\Lambda = L\bbR^k$, where $L \in \bbR^{2 \times k}$ is a matrix with linearly independent columns and $k \in \{1,2\}$.} $\Lambda \subset \bbR^2$, which we refer to as \emph{sampled Gabor phase retrieval}.

% Our interest in the Gabor phase retrieval problem stems from its connection to applications in audio processing such as the phase vocoder \cite{prusa2017phase}. Phase retrieval originated in a different field, however: we can trace its roots back to the early 20$^\mathrm{th}$ century when von Laue discovered the diffraction of X-rays by crystals.

The first uniqueness result for sampled Gabor phase retrieval was presented in \cite{alaifari2021uniqueness} where it is shown that real-valued, bandlimited, square-integrable signals $f$ with bandwidth $B > 0$ can be recovered (up to a constant global sign factor) from $\abs{\calG f}$ sampled on $\tfrac{1}{4B} \bbZ \times \{0\}$. Just two months later, the paper \cite{grohs2023injectivity} appeared in which it is revealed that the real-valuedness assumption from \cite{alaifari2021uniqueness} can be dropped: specifically, the bandlimited square-integrable signals $f$ whose Fourier transform is in $L^4$ can be recovered (up to a constant global phase factor) from $\abs{\calG f}$ sampled on $\tfrac{1}{4B} \bbZ \times \bbZ$. Another two months later, the paper \cite{alaifari2021phase} appeared, showing that the bandlimitedness assumption cannot be dropped. This is achieved by constructing two square-integrable functions that do not agree (up to a constant global phase factor) but whose Gabor transform magnitudes agree on $\Lambda$, where $\Lambda \subset \bbR^2$ is a general lattice in the time-frequency plane.

Finally, we mention the recent preprint \cite{wellershoff2021injectivity} in which results from \cite{grohs2023injectivity} are extended to show that all bandlimited square-integrable signals $f$ can be recovered (up to a constant global phase factor) from $\abs{\calG f}$ sampled on $\tfrac{1}{4B} \bbZ \times \bbN$. In this paper, we further strengthen this result and prove that in fact bandlimited square-integrable signals $f$, with bandwidth $B > 0$, can be recovered (up to a constant global phase factor) from $\abs{\calG f}$ sampled on $\tfrac{1}{4B} \bbZ \times \{\omega_0,\omega_1\}$, where $\omega_0 \neq \omega_1$. As $\tfrac{1}{4B}$ is exactly twice the Nyquist rate, we therefore show that sampling at twice the Nyquist rate in two frequency bins guarantees uniqueness in Gabor phase retrieval as advertised in the title. We point out that a similar result was already known for the Cauchy wavelet transform \cite{alaifari2020phase}: more precisely, sampling at twice the Nyquist rate at two scales guarantees uniqueness in Cauchy wavelet phase retrieval.

\begin{remark}
    The original motivation for this paper stems from a resemblance of the result in \cite{alaifari2021uniqueness} on Gabor sign retrieval with the work in \cite{balan2006signal} on finite-dimensional sign retrieval: in the prior, it is shown that sampling at twice the Nyquist rate in a single frequency bin guarantees uniqueness in Gabor sign retrieval while, in the latter, it is shown that $2n-1$ generic measurement vectors are necessary and sufficient for uniqueness in finite-dimensional sign retrieval. As it is also known that on the order of $4n$ generic measurement vectors are necessary and sufficient for uniqueness in finite-dimensional phase retrieval \cite{balan2006signal,heinosaari2013quantum}, it seems natural to ask whether sampling at four times the Nyquist rate in one frequency bin or at twice the Nyquist rate in two frequency bins would guarantee uniqueness in Gabor phase retrieval. The former is clearly untrue as can be seen from considering $f,g \in \mathrm{PW}_B^2$ real-valued and 
    \begin{equation*}
        \abs{\calG (f + \rmi g)} = \abs{\calG (f - \rmi g)} \mbox{ on } \bbR \times \{0\}.
    \end{equation*}
    The latter is shown to be true in this paper.
\end{remark}

\subsection*{Notation} We denote the normalised Gaussian by $\phi(t) = \sqrt[4]{2} \exp(-\pi t^2)$, where $t \in \bbR$. For $-\infty < a < b < \infty$ and $f \in L^p(\bbR)$, with $p \in [1,\infty)$, we write $\operatorname{supp} f \subset [a,b]$ if $f(t) = 0$, for a.e.~$t \not\in [a,b]$. Moreover, we define the families of \emph{translation} and \emph{modulation} operators $(\operatorname{T}_x)_{x \in \bbR} : L^p(\bbR) \to L^p(\bbR)$ and $(\operatorname{M}_\omega)_{\omega \in \bbR} : L^p(\bbR) \to L^p(\bbR)$ by 
\begin{equation*}
    \operatorname{T}_x f(t) := f(t-x), \qquad \operatorname{M}_\omega f(t) := f(t) \rme^{2\pi\rmi t \omega}.
\end{equation*}
We furthermore use the convention 
\begin{equation*}
    \calF f (\xi) = \int_\bbR f(t) \rme^{-2\pi\rmi t \xi} \dd t, \qquad \xi \in \bbR,
\end{equation*}
for the Fourier transform of $f \in L^1(\bbR) \cup L^2(\bbR)$ and note that the Fourier transform of the normalised Gaussian is the normalised Gaussian itself, i.e.~$\calF \phi = \phi$. Additionally, as is usual in the phase retrieval literature, we will introduce an equivalence relation on the set of complex-valued functions via 
\begin{equation*}
    f \sim g :\iff \exists \alpha \in \bbR: f = \rme^{\rmi \alpha} g
\end{equation*}
and say that $f$ and $g$ \emph{agree up to global phase} if $f \sim g$. The space of polynomials with complex argument and complex coefficients is denoted by $\bbC[z]$. Similarly, the subspace of degree $n \in \bbN$ polynomials in $\bbC[z]$ is denoted by $\bbC_n[z]$. Finally, for an entire function $F : \bbC \to \bbC$, we denote its zero set by $Z(F) \subset \bbC$ and define a function $m_F : \bbC \to \bbN_0$ which assigns the multiplicity of $z$ as a zero of $F$ to every $z \in \bbC$. Note that we use the convention $m_F(z) = 0$, for $z \not\in Z(F)$.

\subsection*{Definitions and basic notions} 

We will work with the \emph{Paley--Wiener spaces} of bandlimited functions 
\begin{equation*}
    \mathrm{PW}_B^p := \{ f \in L^p(\bbR) \,;\, \operatorname{supp} (\calF f) \subset [-B,B] \}, \qquad p \in [1,\infty],
\end{equation*}
where $B > 0$. With this definition, it is well known that $\mathrm{PW}_B^1 \subset \mathrm{PW}_B^2$. It turns out to be useful to consider the \emph{Bargmann transform} of square-integrable signals $f \in L^2(\bbR)$ given by 
\begin{equation*}
    \calB f (z) := \sqrt[4]{2} \int_\bbR f(t) \rme^{2\pi t z - \pi t^2 - \tfrac{\pi}{2}z^2} \dd t, \qquad z \in \bbC. 
\end{equation*}
One of the two main reasons for this is that the Bargmann transform and the Gabor transform (as defined in equation~\eqref{eq:Gabor_transform}) are related via the formula \cite[Proposition 3.4.1 on p.~54]{grochenig2001foundations}
\begin{equation*}
    \calG f (x,-\omega) = \rme^{\pi \rmi x \omega} \calB f(x + \rmi \omega) \rme^{-\tfrac{\pi}{2} \left( x^2 + \omega^2 \right)}, \qquad (x,\omega) \in \bbR^2.
\end{equation*}
The other main reason is that the Bargmann transform of a square-integrable signal is an entire function of finite order. More precisely, we define the \emph{order} of an entire function $F : \bbC \to \bbC$ to be
\begin{equation*}
    \rho := \limsup_{r \to \infty} \frac{\log(\log (\sup_{\abs{z} = r} \abs{F(z)}))}{\log r}
\end{equation*}
and say that $F$ is of \emph{finite order} if $\rho < \infty$. For $F = \calB f$, with $f \in L^2(\bbR)$, one can use \cite[Theorem 3.4.2 on p.~54]{grochenig2001foundations} to show that $\rho \leq 2$: we say that $\calB f$ is of \emph{second-order}.

\section{Preliminaries}

Our goal is to recover a bandlimited function $f$ with bandwidth $B$ from the magnitudes of its Gabor transform $\abs{\calG f}$ sampled on the set $\tfrac{1}{4B} \bbZ \times \{\omega_0,\omega_1\}$, where $\omega_0 \neq \omega_1$. In order to do this, we follow a three-step procedure.

\begin{enumerate}
    \item We note that $\abs{\calG f(\cdot,\omega)}^2$ is bandlimited, for all $\omega \in \bbR$, and use the Nyquist--Shannon sampling theorem to recover $\abs{\calG f}^2$ on $\bbR \times \{\omega_0,\omega_1\}$. 
    \item Relating this to the Bargmann transform of $f$ shows that it suffices to analyse the recovery of a second-order entire function from magnitude measurements on two parallel lines. In this direction, we show that a second-order entire function is either uniquely determined (up to global phase) by its magnitude on two parallel lines or it has infinitely many evenly spaced zeroes.
    \item Finally, we make use of the bandlimitedness of $f$ again and show that the Bargmann transform of $f$ can only have infinitely many evenly spaced zeroes if $f = 0$.
\end{enumerate}

Let us start with the realisation of item 1. We note that the following lemma already follows from the considerations in \cite{wellershoff2021injectivity}. A considerably simpler proof based on a different convention for the Paley--Wiener spaces is given here. 

\begin{lemma}\label{lem:gabor_bandlimited}
    Let $B > 0$, $\omega \in \bbR$ and $f \in \mathrm{PW}_B^2$. Then, $x \mapsto \abs{\calG f(x,\omega)}^2 \in \mathrm{PW}_{2B}^2$.
\end{lemma}

\begin{proof}
    We have $x \mapsto \rme^{2\pi\rmi x\omega} \calG f(x,\omega) \in \mathrm{PW}_B^2$ since it is the (inverse) Fourier transform of $\calF f \cdot \operatorname{T}_\omega \phi \in L^2(\bbR)$ which satisfies $\supp(\calF f \cdot \operatorname{T}_\omega \phi) \subset [-B,B]$ \cite[Equation (3.5) on p.~39]{grochenig2001foundations}. Therefore, $x \mapsto \abs{\calG f(x,\omega)}^2 \in L^1(\bbR)$ and applying the Fourier convolution theorem to
    \begin{equation*}
        x \mapsto \abs{\calG f(x,\omega)}^2 = \calG f(x,\omega) \overline{\calG f(x,\omega)}
    \end{equation*}
    shows that $x \mapsto \abs{\calG f(x,\omega)}^2 \in \mathrm{PW}_{2B}^1 \subset \mathrm{PW}_{2B}^2$.
\end{proof}

Next, we move on to item 2 of our three-step procedure: the analysis of the recovery of a second-order entire function from magnitude measurements on two parallel lines. Interestingly, it can be shown that it is impossible to reconstruct a finite order entire function (up to global phase) from magnitude information on two parallel lines \cite{wellershoff2022phase}. It is therefore also impossible to recover a square-integrable signal from Gabor phase retrieval measurements on two parallel lines; a fact which has been used to construct the counterexamples to Gabor phase retrieval in \cite{alaifari2021phase}. We are not considering general square-integrable signals here however and the bandlimitedness assumption turns out to be sufficient to exclude counterexamples.

Let us start by realising that fixing the magnitude of an entire function on two parallel lines enforces a periodicity in its zeroes. More precisely, we show that, if two entire functions $F$ and $G$ have magnitudes that agree on $\bbR \cup (\bbR + \rmi \tau)$, then $m_F - m_G$ is $(2 \rmi \tau)$-periodic.

\begin{lemma}\label{lem:periodicity_zeroes}
    Let $\tau > 0$ and let $F,G \in \bbC \to \bbC$ be two entire functions such that $\abs{F} = \abs{G}$ on $\bbR \cup (\bbR + \rmi \tau)$. Then, 
    \begin{equation*}
        m_F(z + 2 \rmi \tau) - m_G(z + 2 \rmi \tau) = m_F(z) - m_G(z), \qquad z \in \bbC.
    \end{equation*}
\end{lemma}

\begin{proof}
    Let $z \in \bbC$ denote an arbitrary complex number. According to \cite[Proposition 1 on p.~261]{mcdonald2004phase}, $\abs{F} = \abs{G}$ on $\bbR$ implies that 
    \begin{equation*}
        F(z) \overline{F(\overline z)} = G(z) \overline{G(\overline z)}.
    \end{equation*}
    Therefore, after looking at the zeroes of the above equation exclusively, we have 
    \begin{equation*}
        m_F(z) + m_F(\overline z) = m_G(z) + m_G(\overline z).
    \end{equation*}
    The same argument applied to $F_\tau(z) := F(z+\rmi \tau)$ and $G_\tau(z) := G(z + \rmi \tau)$ yields 
    \begin{equation*}
        m_F(z + \rmi \tau) + m_F(\overline z + \rmi \tau) = m_G(z + \rmi \tau) + m_G(\overline z + \rmi \tau)
    \end{equation*}
    such that we can conclude that
    \begin{equation*}
        m_F(z + 2 \rmi \tau) - m_G(z + 2 \rmi \tau) = m_G(\overline z) - m_F(\overline z) = m_F(z) - m_G(z).
    \end{equation*}
\end{proof}

The periodicity in $m_F-m_G$ directly implies that the zeroes (with multiplicities) of $F$ and $G$ agree everywhere if they agree on the strip $\bbR + \rmi (-\tau,\tau]$. Combining this insight with the Hadamard factorisation theorem yields that a second-order entire function is either uniquely determined (up to global phase) by its magnitude on the two parallel lines $\bbR \cup (\bbR + \rmi \tau)$ or that it has at least one zero in the strip $\bbR + \rmi (-\tau,\tau]$.

\begin{corollary}\label{cor:agree_utgp}
    Let $\tau > 0$ and let $F,G \in \bbC \to \bbC$ be two entire functions \emph{of second-order} such that $\abs{F} = \abs{G}$ on $\bbR \cup (\bbR + \rmi \tau)$. If $m_F - m_G = 0$ on $\bbR + \rmi (-\tau,\tau]$, then $F \sim G$.
\end{corollary}

\begin{proof}
    If $m_F - m_G = 0$ on $\bbR + \rmi (-\tau,\tau]$, then Lemma~\ref{lem:periodicity_zeroes} implies $m_F - m_G = 0$ such that the zeroes (with multiplicity) of $F$ and $G$ agree. It therefore follows from Hadamard's factorisation theorem (cf.~\cite[Subsection~8.24 on p.~250]{titchmarsh1939theory}) that 
    \begin{equation*}
        F(z) = \mathrm{e}^{Q(z)} G(z), \qquad z \in \bbC,
    \end{equation*}
    where $Q \in \bbC_2[z]$ is a quadratic polynomial. As $\abs{F} = \abs{G}$ on $\bbR$, we furthermore have 
    \begin{equation*}
        \abs{F(x)} = \rme^{\Re Q(x)} \abs{G(x)} = \rme^{\Re Q(x)} \abs{F(x)}, \qquad x \in \bbR,
    \end{equation*}
    which implies that $\exp(\Re Q) = 1$ holds on $\bbR \setminus Z(F)$. As $F$ is entire, $Z(F)$ has no accumulation point in $\bbR$ and is thus of measure zero. It follows that $\Re Q = 0$ almost everywhere (and thus everywhere) on $\bbR$. We can now write 
    \begin{equation*}
        Q(z) = \rmi \left(\alpha + \lambda_1 z + \lambda_2 z^2\right), \qquad z \in \bbC,
    \end{equation*}
    for some $\alpha,\lambda_1,\lambda_2 \in \bbR$. The same argument as above shows that $\abs{F} = \abs{G}$ on $\bbR + \rmi \tau$ implies $\Re Q = 0$ on $\bbR + \rmi \tau$. Therefore, we find 
    \begin{equation*}
        \Re Q(x + \rmi \tau) = \lambda_1 \tau + 2 \lambda_2 \tau x = 0, \qquad x \in \bbR,
    \end{equation*}
    which proves that $\lambda_1 = \lambda_2 = 0$ and thus $F = \rme^{\rmi \alpha} G$. 
\end{proof}

\begin{remark}
    Corollary~\ref{cor:agree_utgp} actually holds for general entire functions of finite order $F,G \in \bbC \to \bbC$. The proof remains mostly the same: only the polynomial $Q \in \bbC[z]$ is of arbitrary order instead of quadratic. We therefore have 
    \begin{equation*}
        Q(z) = \rmi \sum_{j = 0}^n \lambda_j z^j, \qquad z \in \bbC,
    \end{equation*}
    for some $n \in \bbN$ and $(\lambda_j)_{j=0}^n \in \bbR$, and, as $\Re Q = 0$ on $\bbR + \rmi \tau$, we find
    \begin{equation*}
        \sum_{j = 0}^n \lambda_j \Im\left( (x + \rmi \tau)^j \right) = 0, \qquad x \in \bbR.
    \end{equation*}
    The binomial theorem, then shows
    \begin{equation*}
        0 = \sum_{j = 0}^n \lambda_j \sum_{k = 0}^j \binom{j}{k} x^{k} \tau^{j-k} \Im\left(\rmi^{j-k}\right)
        = \sum_{k = 0}^n x^k \sum_{j = k}^n \binom{j}{k} \lambda_j \tau^{j-k} \Im\left(\rmi^{j-k}\right),
    \end{equation*}
    for $x \in \bbR$. Comparing coefficients yields 
    \begin{align*}
        0 &= \sum_{j = k}^n \binom{j}{k} \lambda_j \tau^{j-k} \Im\left(\rmi^{j-k}\right) = \sum_{j = 0}^{n-k} \binom{j+k}{k} \lambda_{j+k} \tau^j \Im\left(\rmi^j\right) \\
        &= \sum_{j = 0}^{\lfloor (n-k-1)/2 \rfloor} \binom{2j+k+1}{k} \lambda_{2j+k+1} \tau^{2j+1} (-1)^j,
    \end{align*}
    for $k = 0,\dots,n-1$, which can be used inductively to show that $\lambda_j = 0$, for $j = 1,\dots,n$. Indeed, for $k = n-1$, we find 
    \begin{equation*}
        0 = \binom{n}{n-1} \lambda_n \tau \implies \lambda_n = 0.
    \end{equation*}
    If we assume that $\lambda_j = 0$, for $j = k+2,\dots,n$, with $k \in \{0,\dots,n-2\}$, then
    \begin{multline*}
        0 = \sum_{j = 0}^{\lfloor (n-k-1)/2 \rfloor} \binom{2j+k+1}{k} \lambda_{2j+k+1} \tau^{2j+1} (-1)^j = \binom{k+1}{k} \lambda_{k+1} \tau \\
        \implies \lambda_{k+1} = 0.
    \end{multline*}
\end{remark}

We can finally turn to item 3 of our three step procedure: using bandlimitedness to show that $f = 0$ if $\calB f$ has infinitely many evenly spaced zeroes. The inspiration for this step actually comes from \cite{grohs2023injectivity,wellershoff2021injectivity}, where a M{\"u}ntz--Sz{\'a}sz type result from \cite{zalik1978approximation} is used in order to recover bandlimited $f$ from $\abs{\calG f}$ on $\bbR \times \bbN$. We rely on a slight generalisation of that same M{\"u}ntz--Sz{\'a}sz type result for this paper.

\begin{theorem}[Zalik's theorem; cf.~Theorem~4 in \cite{zalik1978approximation}]\label{thm:zalik}
    Let $p \in [1,\infty)$, let $a,b \in \bbR$ be such that $a < b$, let $r > 0$ and let $(c_n)_{n \in \bbN} \in \bbC$ be a sequence of distinct complex numbers such that there exists a $\delta > 0$ and an $N_0 \in \bbN$ with 
    \begin{equation*}
        \abs{\Re \left[ c_n - \tfrac12 \right]} \geq \delta \abs{c_n - \tfrac12}, \qquad n \geq N_0.
    \end{equation*}
    Then, 
    \begin{equation*}
        \left\{ t \mapsto \rme^{-r^2(t - c_n)^2} \,;\, n \in \bbN \right\}
    \end{equation*}
    is complete in $L^p([a,b])$ if and only if 
    \begin{equation*}
        \sum_{n \in \bbN,~c_n \neq 0} \abs{c_n}^{-1}
    \end{equation*}
    diverges.
\end{theorem}

\begin{proof}
    The theorem follows from the original proof in \cite{zalik1978approximation} with some small modifications.
\end{proof}

\begin{remark}
    Zalik's original result is stated for $p = 2$ and \emph{real numbers} $(c_n)_{n \in \bbN}$. We do not make use of the added generality in the integrability parameter $p$ here but the proof of our main result does require the sequence $(c_n)_{n \in \bbN}$ to be complex-valued.
\end{remark}

\section{The main result}

We are now in a position to state and prove our main result: general bandlimited signals can be recovered from their Gabor transform magnitudes sampled at twice the Nyquist rate in two frequency bins.

\begin{theorem}[Main result]\label{thm:main_result}
    Let $B > 0$, let $\omega_0,\omega_1 \in \bbR$ be such that $\omega_0 < \omega_1$ and let $f,g \in \mathrm{PW}_B^2$. Then, $f \sim g$ if (and only if) $\abs{\calG f} = \abs{\calG g}$ on $\tfrac{1}{4B} \bbZ \times \{\omega_0,\omega_1\}$.
\end{theorem}

\begin{proof}
    Let us suppose that $\abs{\calG f} = \abs{\calG g}$ on $\tfrac{1}{4B} \bbZ \times \{\omega_0,\omega_1\}$ and fix $j \in \{0,1\}$. According to Lemma~\ref{lem:gabor_bandlimited}, the functions 
    \begin{equation*}
        x \mapsto \lvert \calG f (x,\omega_j) \rvert^2, \qquad x \mapsto \lvert \calG g (x,\omega_j) \rvert^2
    \end{equation*}
    are in $\mathrm{PW}_{2B}^2$. By the Nyquist--Shannon sampling theorem, we therefore find that 
    \begin{equation}
        \label{eq:two_parallel_lines}
        \abs{\calG f (x,\omega_j)}^2 = \abs{\calG g (x,\omega_j)}^2, \qquad x \in \bbR.
    \end{equation}

    Next, we define the second-order entire functions $F(z) := \calB f (z + \rmi \omega_0)$ and $G(z) := \calB g (z + \rmi \omega_0)$, for $z \in \bbC$, which satisfy
    \begin{equation*}
        \abs{F(x)} = \abs{G(x)}, \qquad \abs{F(x + \rmi \tau)} = \abs{G(x + \rmi \tau)}, \qquad x \in \bbR,
    \end{equation*}
    for $\tau := \omega_1-\omega_0 > 0$, according to equation~\eqref{eq:two_parallel_lines}. In the rest of this proof, we will distinguish between two cases: $m_F - m_G = 0$ on $\bbR \times \rmi(-\tau,\tau]$ and $m_F - m_G \neq 0$ on $\bbR \times \rmi(-\tau,\tau]$. In the first case, i.e.~when $m_F - m_G = 0$ on $\bbR \times \rmi(-\tau,\tau]$, Corollary~\ref{cor:agree_utgp} implies that $F \sim G$. Since the Bargmann transform is injective\footnote{Actually, the range of the Bargmann transform can be equipped with an inner product and thereby turned into a Hilbert space which known as the Fock space $\calF^2(\bbC)$. The Bargmann transform then turns out to be a unitary operator mapping $L^2(\bbR)$ onto $\calF^2(\bbC)$ \cite[Theorem 3.4.3 on p.~56]{grochenig2001foundations}.}, it follows that $f \sim g$.

    In the second case, we may without loss of generality assume that there exists a complex number $z_0 \in \bbR \times (-\tau,\tau]$ such that $m_F(z_0) - m_G(z_0) > 0$: indeed, there exists a complex number at which $m_F - m_G$ is non-zero and we can exchange $F$ and $G$ if we have $m_F(z_0) - m_G(z_0) < 0$. By Lemma~\ref{lem:periodicity_zeroes}, we therefore find that 
    \begin{equation*}
        m_F(z_0 + 2 \rmi k \tau) - m_G(z_0 + 2 \rmi k \tau) = m_F(z_0) - m_G(z_0) > 0,
    \end{equation*}
    for $k \in \bbZ$. It follows that $(z_0 + 2 \rmi k \tau)_{k \in \bbZ} \in \bbC$ forms a sequence of zeroes of $F$, i.e.
    \begin{equation*}
        \calB f (z_0 + \rmi (\omega_0 + 2 k \tau) ) = 0, \qquad k \in \bbZ. 
    \end{equation*}
    The Bargmann transform satisfies the nice symmetry $\calB f (- \rmi z) = \calB \calF f(z)$, for $z \in \bbC$ (c.f.~\cite[Equation (3.10a) on p.~207]{bargmann1961Hilbert}). Therefore, 
    \begin{align*}
        0 &= \calB f (z_0 + \rmi (\omega_0 + 2 k \tau) ) = \calB \calF f (\rmi z_0 - \omega_0 - 2 k \tau) \\
        % &= \int_\bbR \calF f(\xi) \rme^{2 \pi \xi (\omega_0 + 2 k \tau - \rmi z_0) - \pi \xi^2 - \tfrac{\pi}{2}(\omega_0 + 2 k \tau - \rmi z_0)^2} \dd \xi \\
        &= \rme^{\tfrac{\pi}{2} (\rmi z_0 - \omega_0 - 2 k \tau)^2} \int_{-B}^B \calF f(\xi) \rme^{-\pi (\xi - \rmi z_0 + \omega_0 + 2 k \tau)^2} \dd \xi
    \end{align*}
    which implies that $\calF f$ is orthogonal to $\xi \mapsto \rme^{-\pi (\xi - \rmi z_0 + \omega_0 + 2 k \tau)^2}$ in $L^2([-B,B])$, for all $k \in \bbZ$. Theorem~\ref{thm:zalik} now implies that $\calF f = 0$ and thus $f = 0$. Hence, $F = 0$ and as $\abs{F} = \abs{G}$ the identity theorem of complex analysis can be used to show that $G = 0$ such that $g = 0 = f$.
\end{proof}

Our main result may alternatively be stated without reference to the Gabor transform. Indeed, it is equivalent to the following theorem.

\begin{theorem}
    Let $B > 0$ and let $\omega_0, \omega_1 \in \bbR$ be such that $\omega_0 \neq \omega_1$. If $f,g \in \mathrm{PW}_B^2$ satisfy
    \begin{equation*}
        \abs{\phi \ast (\operatorname{M}_{-\omega_j} f)} = \abs{\phi \ast (\operatorname{M}_{-\omega_j} g)}, \qquad j = 0,1,
    \end{equation*}
    then $f \sim g$.
\end{theorem}


\bibliography{sn-bibliography}% common bib file


\end{document}